\documentclass[10pt]{amsart}

\usepackage[margin=1.5in]{geometry}
\usepackage{mathtools}
\usepackage{amssymb}
\usepackage{enumitem}
\usepackage{hyperref}
\usepackage[nameinlink,noabbrev]{cleveref}

\hypersetup{
  hidelinks,
  pdftitle={Finite-Core Exhaustion and the Failure of Strict Sequential-Pressure Approximation in Coded Shifts},
  pdfauthor={C. Evans Hedges},
  pdfsubject={Coded shifts, topological pressure, and computability},
  pdfkeywords={Coded shifts, topological pressure, sequential pressure, computable analysis, equilibrium states}
}

\allowdisplaybreaks

\newtheorem{theorem}{Theorem}[section]
\newtheorem{proposition}[theorem]{Proposition}
\newtheorem{lemma}[theorem]{Lemma}
\newtheorem{corollary}[theorem]{Corollary}
\newtheorem{definition}[theorem]{Definition}
\theoremstyle{remark}
\newtheorem{remark}[theorem]{Remark}
\newtheorem{example}[theorem]{Example}

\newcommand{\A}{\mathcal A}
\newcommand{\G}{\mathcal G}
\newcommand{\Lcal}{\mathcal L}
\newcommand{\M}{\mathcal M}
\newcommand{\Ecal}{\mathcal E}
\newcommand{\Sigmafull}{\A^{\mathbb Z}}
\newcommand{\htop}{h_{\mathrm{top}}}
\newcommand{\hcon}{h_{\mathrm{con}}}
\newcommand{\hres}{h_{\mathrm{res}}}
\newcommand{\Xcon}{X_{\mathrm{con}}}
\newcommand{\Xres}{X_{\mathrm{res}}}
\newcommand{\eps}{\varepsilon}

\title[Finite-core exhaustion and strict pressure]{Finite-Core Exhaustion and
the Failure of\\Strict Sequential-Pressure Approximation in Coded Shifts}
\author{C. Evans Hedges}
\thanks{Independent researcher.}
\email{evans@hedgesfamily.com}
\date{8 September 2026}

\begin{document}

\begin{abstract}
Let \(X=X(\G)\) be a coded shift whose generating set uniquely represents its
concatenation set.  We prove that, for every continuous potential \(\varphi\),
the pressures of the finite-generator subshifts (the finite cores) converge
to the sequential pressure \(P_{\rm seq}(\varphi,\G)\), the supremum of free
energy over invariant measures giving full mass to the concatenation set.
Thus full sequential pressure is the exact condition for the finite cores to
recover global pressure.  The proof is a refinement of the inducing argument
of Burr, Das, Wolf, and Yang.

Given an arbitrary enumeration of the generators and a length-ordered
enumeration of the language, we obtain an algorithm which computes the global
pressure of every computable potential with full sequential pressure.  The set
of equilibrium states is recursively compact, and a unique equilibrium state
is computable from the same data.  For the zero potential,
\(h_{\rm con}\geq h_{\rm res}\) therefore implies computability of topological
entropy and of any unique measure of maximal entropy.  We also construct a
uniquely represented presentation of the full binary shift for which the zero
potential has full sequential pressure but is not a uniform limit of
potentials satisfying \(P_{\rm seq}>P_{\rm res}\), disproving a conjecture of
Burr, Das, Wolf, and Yang.
\end{abstract}

\keywords{Coded shifts, topological pressure, sequential pressure,
computable analysis, equilibrium states}
\subjclass[2020]{Primary 37B10; Secondary 37D35, 03D78}

\maketitle
\enlargethispage{3pt}

\section{Introduction}

Coded shifts, introduced by Blanchard and Hansel
\cite{BlanchardHanselCoded}, are specified by countable collections of finite
words, but the shift itself is the closure of their bi-infinite
concatenations.  The finite subcollections generate an increasing family of
sofic subshifts, which we call the finite cores.  The natural approximation
problem is to identify the limit of the finite-core pressures and determine
when that limit is the pressure of the full coded shift.

The closure may contain residual points which are not themselves
concatenations, and invariant measures carried by those points can have
free energy larger than that of every finite core.  Thus the finite-core limit
need not be the global pressure.  Our first result identifies the finite-core
limit exactly with sequential pressure.  Our second shows that the resulting
equality regime is genuinely larger than the closure of the
strict-pressure-gap regime: a uniquely represented presentation of the full
binary shift disproves a conjecture of Burr, Das, Wolf, and Yang.  The pressure
identity also turns finite-core approximation into a computability theorem.
In particular, under the effective hypotheses below, a unique equilibrium
state is computable from the same data even when it gives full mass to the
residual set.

\subsection{Coded presentations}

Fix a finite alphabet \(\A\) and a nonempty countable collection \(\G\) of
nonempty words over \(\A\).  The concatenation set \(\Xcon(\G)\) consists of
the points obtained by bi-infinite concatenation of words in \(\G\).  The coded
shift and the residual set of this presentation are
\[
 X=X(\G)=\overline{\Xcon(\G)},
 \qquad
 \Xres(\G)=X\setminus\Xcon(\G).
\]
We say that \(\G\) \emph{uniquely represents} \(\Xcon(\G)\) if every point in
\(\Xcon(\G)\) has a unique decomposition into generators, including a unique
position of the origin within its generator.

Choose any enumeration \(\G=\{g_1,g_2,\ldots\}\), with the list ending when
\(\G\) is finite.  Let \(F_m\) consist of the first \(m\) generators, or all
of \(\G\) if the list has already ended, and set
\[
 X_m=X(F_m).
\]
If \(\G\) is finite, this sequence is therefore constant at \(X\) after the
last generator.  We call the subshifts \(X_m\) the \emph{finite cores}.  The
effective presentation we use consists of the generator enumeration, in the
order supplied, and an enumeration of the language \(\Lcal(X)\), the set of finite
words occurring in \(X\), in nondecreasing order of word length.  All
computability statements below are uniform in these oracles and in the
additional computable data named in the statement.  Uniformity is understood
on the class of inputs satisfying the stated hypotheses.

Let \(\M_\sigma(X)\) be the invariant Borel probability measures on \(X\).  For
\(\varphi\in C(X)\), define
\[
 P_X(\varphi)=\sup_{\mu\in\M_\sigma(X)}
 \left(h_\mu(\sigma)+\int\varphi\,d\mu\right).
\]
An invariant measure attaining this supremum is an equilibrium state.  The
sequential and residual pressures are
\begin{align*}
 P_{\rm seq}(\varphi,\G)
 &=\sup\left\{h_\mu(\sigma)+\int\varphi\,d\mu:
                 \mu\in\M_\sigma(X),\
                 \mu(\Xcon(\G))=1\right\},\\
 P_{\rm res}(\varphi,\G)
 &=\sup\left\{h_\mu(\sigma)+\int\varphi\,d\mu:
                 \mu\in\M_\sigma(X),\
                 \mu(\Xres(\G))=1\right\}.
\end{align*}
We call the measures in the first class \emph{sequential} and those in the
second class \emph{residual}.  We define the supremum of an empty class to be
\(-\infty\).  Unique representation makes
\(\Xcon(\G)\) a Borel set, and both \(\Xcon(\G)\) and \(\Xres(\G)\) are
invariant.  Ergodic decomposition therefore gives
\[
 P_X(\varphi)=\max\{P_{\rm seq}(\varphi,\G),
                         P_{\rm res}(\varphi,\G)\}.
\]
We say that \(\varphi\) has \emph{full sequential pressure} when
\(P_X(\varphi)=P_{\rm seq}(\varphi,\G)\).  For \(\varphi=0\), write
\[
 \hcon(X,\G)=P_{\rm seq}(0,\G),
 \qquad
 \hres(X,\G)=P_{\rm res}(0,\G).
\]

Every invariant measure on \(X_m\) gives full mass to \(\Xcon(\G)\).  Thus the
pressure of a finite core is bounded above by sequential pressure.  Our first
result shows that these pressures increase to sequential pressure,
independently of the chosen enumeration:

\begin{theorem}[Finite cores exhaust sequential pressure]
\label{thm:sequential-exhaustion}
Suppose that \(\G\) uniquely represents \(\Xcon(\G)\).  Then, for every
\(\varphi\in C(X)\),
\[
 P_{\rm seq}(\varphi,\G)
 =\sup_{m\geq1}P_{X_m}(\varphi)
 =\lim_{m\to\infty}P_{X_m}(\varphi).
\]
Consequently,
\[
 P_X(\varphi)=\lim_{m\to\infty}P_{X_m}(\varphi)
 \quad\Longleftrightarrow\quad
 P_X(\varphi)=P_{\rm seq}(\varphi,\G),
\]
so the finite cores recover global pressure if and only if \(\varphi\) has full
sequential pressure.  In particular, the limit is independent of the
enumeration of \(\G\).
\end{theorem}

The proof follows the inducing argument of Burr, Das, Wolf, and Yang.  They
induce on generator boundaries, approximate the pressure of the resulting
countable full shift by finite-alphabet full shifts, and lift the
finite-alphabet measures back to finite cores
\cite[Theorem~27 and Proposition~28]{BDWYpressure}.  Their application begins
with a locally constant potential having an equilibrium state supported on the
concatenation set.  We apply the same argument to an arbitrary sequential
measure, normalize by its free energy, and then take the supremum over all
sequential measures.  This identifies the finite-core limit for every
continuous potential as sequential pressure.

Following \cite{BDWYpressure}, let \(\operatorname{FSP}(X,\G)\) be the class of
potentials with full sequential pressure, let
\(\operatorname{SSP}(X,\G)\) be the class for which
\(P_{\rm seq}>P_{\rm res}\), and let
\(\operatorname{FSSP}(X,\G)\) be the uniform closure of
\(\operatorname{SSP}(X,\G)\).  Burr, Das, Wolf, and Yang prove
\(\operatorname{FSSP}\subseteq\operatorname{FSP}\) and conjecture the reverse
inclusion \cite[Theorem~29]{BDWYpressure}.  Our second result disproves this
conjecture, even for a presentation of the full shift.

\begin{theorem}[FSP can be strictly larger than FSSP]
\label{thm:dyck-separation}
There is a uniquely represented coded presentation \((X,\G_D)\), with
\(X=\{+,-\}^{\mathbb Z}\), such that
\[
 P_{\rm res}(\varphi,\G_D)=P_X(\varphi)
 \qquad\text{for every }\varphi\in C(X),
\]
whereas
\[
 P_{\rm seq}(0,\G_D)=\htop(X)=\log 2.
\]
Consequently,
\[
 \operatorname{SSP}(X,\G_D)
 =\operatorname{FSSP}(X,\G_D)=\varnothing,
 \qquad
 0\in\operatorname{FSP}(X,\G_D),
\]
and hence \(\operatorname{FSP}(X,\G_D)\neq
\operatorname{FSSP}(X,\G_D)\).
\end{theorem}

The construction and proof are given in \Cref{sec:dyck-separation}.

\subsection{Computability and consequences}

A continuous potential \(\varphi:\A^{\mathbb Z}\to\mathbb R\) is
\emph{computable} if it can be approximated effectively and uniformly by
rational locally constant potentials.  A compact subset of a computable
metric space is \emph{recursively compact} if an algorithm halts exactly when
a given finite family of basic open balls covers it.  In the computable space
of probability measures, a recursively compact singleton is a computable
measure; equivalently here, its cylinder probabilities can be approximated
effectively.
We use the standard computable-metric-space framework; see
\cite[Section~2.1]{HedgesPavlovpressure} for a more detailed introduction in a
closely related symbolic setting, and \cite{HoyrupRojas} for the underlying
computable measure-theoretic framework.

Burr, Das, Wolf, and Yang obtain decreasing upper bounds for global pressure
from the language oracle.  The computable finite-core pressures give
increasing lower bounds.  By \Cref{thm:sequential-exhaustion}, the lower bounds
converge to global pressure whenever the potential has full sequential
pressure.  We therefore obtain:

\begin{corollary}[Computability under full sequential pressure]
\label{cor:pressure-computability}
Let \(X=X(\G)\) be a coded shift over a finite alphabet.  Suppose that \(\G\)
uniquely represents \(\Xcon(\G)\), and that a generator oracle and a
length-ordered language oracle are given.  If
\(\varphi:\A^{\mathbb Z}\to\mathbb R\) is computable and continuous and
\[
 P_X(\varphi)=P_{\rm seq}(\varphi,\G),
\]
then \(P_X(\varphi)\) is computable uniformly from the two oracles and a
computable name for \(\varphi\).  The set of equilibrium states is recursively
compact relative to the same data.  In particular, if \(\varphi\) has a unique
equilibrium state, then it is computable from these data.
\end{corollary}

Taking \(\varphi=0\) gives the following immediate consequence.
\begin{corollary}[Entropy and measures of maximal entropy]
\label{cor:main-coded}
Let \(X=X(\G)\) be a coded shift over a finite alphabet.  Suppose that \(\G\)
uniquely represents \(\Xcon(\G)\), a generator oracle and a length-ordered
language oracle are given, and
\[
  \hcon(X,\G)\geq\hres(X,\G).
\]
Then \(\htop(X)\) is computable uniformly from these oracles, and the set of
measures of maximal entropy is recursively compact relative to them.  If \(X\)
has a unique measure of maximal entropy, then that measure is computable from
these oracles.
\end{corollary}

In the equality examples below, the unique maximizing measure is residual and
the sequential supremum is not attained.  Even so, the generator and language
oracles compute that measure.  Thus finite-word data from the coded
presentation can compute a measure which gives full mass to the residual set.

Under the strict inequality, \Cref{cor:main-coded} also removes the additional
computability hypothesis on the Vere--Jones parameter used by
Kucherenko, L\'opez, and Wolf
\cite[Corollary~1.1]{KLWcomputability}; at equality
it applies whenever the measure of maximal entropy is unique.  The opposite
pressure regime can behave differently: Kucherenko, L\'opez, and Wolf construct
a uniquely represented coded shift \cite[Theorem~D]{KLWcomputability} with
\(\hres(X,\G)>\hcon(X,\G)\) whose unique measure of maximal entropy is not
computable from the supplied generator and language oracles.

The remainder of this paper is organized as follows.
\Cref{sec:preliminaries} contains the symbolic and computability preliminaries.
In \Cref{sec:sequential-boundary} we prove the finite-core pressure identity
and construct the presentation which separates \(\operatorname{FSP}\) from
\(\operatorname{FSSP}\).  Finally, \Cref{sec:effective-consequences} proves the
computability results and gives examples at equality.

\section{Preliminaries}
\label{sec:preliminaries}

\subsection{Coded shifts and symbolic notation}

Fix a finite alphabet \(\A=\{0,\ldots,d-1\}\).  The full two-sided shift is
\(\Sigma=\Sigmafull\) with the left shift \(\sigma\).  We use the metric
\[
 d_\Sigma(x,y)=
 \begin{cases}
  0,&x=y,\\
  2^{-\min\{|j|:x_j\neq y_j\}},&x\neq y.
 \end{cases}
 \tag{2.1}\label{eq:shiftmetric}
\]
A subshift \(X\subset\Sigma\) is a nonempty closed, shift-invariant set.  Its
language is
\[
 \Lcal(X)=\bigcup_{n\geq1}\Lcal_n(X),\qquad
 \Lcal_n(X)=\{x_{[0,n-1]}:x\in X\}.
\]
For \(w\in\A^n\), write
\[
 [w]=\{x\in\Sigma:x_{[0,n-1]}=w\}
\]
for the ambient cylinder; its trace on \(X\) is \(X\cap[w]\).  Let
\(\M_\sigma(X)\) denote the invariant Borel probability measures on \(X\), and
let \(h_\mu(\sigma)\) denote the measure-theoretic (Kolmogorov--Sinai) entropy
of \(\mu\).  We write
\[
 \htop(X)=\sup_{\mu\in\M_\sigma(X)}h_\mu(\sigma).
\]
A measure attaining this supremum is a measure of maximal entropy (MME).

Let \(\G\) be a nonempty countable collection of nonempty finite words over
\(\A\).  The concatenation set is
\[
 \Xcon(\G)=\{\cdots g_{-1}g_0g_1\cdots:g_j\in\G\},
\]
the coded shift is \(X(\G)=\overline{\Xcon(\G)}\), and
\(\Xres(\G)=X(\G)\setminus\Xcon(\G)\).

\begin{definition}[Unique representation]
\label{def:unique-representation}
A \emph{pointed \(\G\)-representation} of \(x\in\Xcon(\G)\) is a pair
\[
 ((g_j)_{j\in\mathbb Z},k),
 \qquad g_j\in\G,\quad 0\leq k<|g_0|,
\]
such that
\[
 x=\sigma^k(\cdots g_{-1}.g_0g_1\cdots),
\]
where the dot marks coordinate zero.  We say that \(\G\) \emph{uniquely
represents} \(\Xcon(\G)\) if every point of \(\Xcon(\G)\) has exactly one
pointed \(\G\)-representation.  Equivalently, \(\G\) is a strong code in the
terminology of \cite{BPRunambiguous}.
\end{definition}

Every coded shift admits some unambiguous code
\cite[Corollary~35]{BPRunambiguous}, but our hypotheses concern the code supplied
by the generator oracle.  In particular, we do not assume that there is an algorithm
which replaces the supplied presentation by a uniquely represented one with
the same oracle information.  When the representation is unique, let
\[
 E_{\G}=\{x\in\Xcon(\G):k=0\}
\]
be the set of points whose distinguished coordinate is a generator boundary.

For a uniquely represented presentation \(X=X(\G)\), following
\cite{KSWergodic}, define
\begin{align*}
 \hcon(X,\G)
 &=
 \sup\{h_\mu(\sigma):\mu\in\M_\sigma(X),\ \mu(\Xcon(\G))=1\},\\
 \hres(X,\G)
 &=
 \sup\{h_\mu(\sigma):\mu\in\M_\sigma(X),\ \mu(\Xres(\G))=1\}.
\end{align*}
By \Cref{lem:boundary-coding}, \(\Xcon(\G)\) and \(\Xres(\G)\) are invariant
Borel sets.  If \(\xi_\mu\) is the ergodic decomposition of \(\mu\), then
\[
 h_\mu(\sigma)=\int h_\nu(\sigma)\,d\xi_\mu(\nu),
\]
while ergodicity gives \(\nu(\Xcon(\G))\in\{0,1\}\).  It follows that
\[
 \htop(X)=\max\{\hcon(X,\G),\hres(X,\G)\}.
\]
Under the strict inequality \(\hcon(X,\G)>\hres(X,\G)\),
\cite[Theorem~B]{KSWergodic} also gives uniqueness of the MME.

\subsection{Effective presentation}

\begin{definition}[Presentation oracles]
\label{def:oracles}
An oracle for \(\Lcal(X)\) enumerates every word in \(\Lcal(X)\) in
nondecreasing order of length.  A generator oracle enumerates every generator
in an arbitrary order; a finite enumeration ends with a designated end marker.
\end{definition}

The language enumeration is ordered by length so that each finite language
level can be computed.  The generator enumeration may be arbitrary because
the proofs use only finite prefixes of that enumeration.

\begin{lemma}[Finite language levels are decidable]
\label{lem:language-decidable}
From the language oracle, one can compute the complete finite set
\(\Lcal_n(X)\), uniformly in \(n\).
\end{lemma}

\begin{proof}
Run the enumeration until the first word of length strictly greater than \(n\)
appears.  Because \(X\) is nonempty, its language contains words of every
length, so this stage is reached.  Nondecreasing length implies that no further
word of length \(n\) can appear.  The distinct length-\(n\) words already seen
are exactly \(\Lcal_n(X)\).
\end{proof}

\begin{remark}
The ordered language enumeration supplies both positive and negative
information at each finite length.  This is stronger than an unordered
positive enumeration.
\end{remark}

\begin{lemma}[Effective presentation]
\label{lem:effective-presentation}
The language oracle uniformly computes a recursively compact name for \(X\),
a dense sequence of computable points in \(X\), and a recursively compact name
for \(\M_\sigma(X)\) in the standard computable space of Borel probability
measures on \(\Sigma\).
\end{lemma}

\begin{proof}
By \Cref{lem:language-decidable}, the forbidden words are effectively
enumerable, and
\[
 \Sigma\setminus X
 =
 \bigcup_{\substack{w\notin\Lcal(X)\\j\in\mathbb Z}}
 \{x:x_{[j,j+|w|-1]}=w\}.
\]
The cylinders on the right are uniformly enumerable clopen sets, each an
effectively computable finite union of basic open balls.  Since \(\Sigma\) is
recursively compact, a finite family \(U\) of basic open sets
covers \(X\) exactly when, at some finite stage, \(U\) together with finitely
many of these forbidden cylinders covers \(\Sigma\).  This gives the required
cover-semidecision procedure for \(X\).

We next construct a computable point of \(X\) in every nonempty cylinder.  Fix
\(w\in\Lcal(X)\) and set \(v_0=w\).
Given \(v_r\), use \Cref{lem:language-decidable} to find
\(a,b\in\A\) for which \(av_rb\in\Lcal(X)\), and set
\(v_{r+1}=av_rb\).  Such an extension exists because \(v_r\) occurs in a
bi-infinite point of \(X\).  Place \(v_r\) on the coordinates
\([-r,|w|+r-1]\).  The resulting cylinders have nonempty, nested intersections
with \(X\), and their intersection is a single point \(x^{(w)}\).  To compute
\(x^{(w)}_j\), construct \(v_r\) until its coordinate interval contains \(j\).
Thus \(x^{(w)}\) is computable, and
\[
 \{\sigma^j x^{(w)}:w\in\Lcal(X),\ j\in\mathbb Z\}
\]
is a countable dense family in \(X\).

The probability measures supported on a recursively compact set form a
recursively compact set
\cite[Lemma~2.5.1 and Proposition~2.4.4]{GHRinvariant}.  The shift
push-forward on measures is computable
\cite[Theorem~3.1.1]{GHRinvariant}.  In any fixed computable metric \(d\) on
measures, the condition
\[
 \sigma_*\mu=\mu
 \quad\Longleftrightarrow\quad
 d(\sigma_*\mu,\mu)=0
\]
defines an effectively closed subset.  It follows that
\(\M_\sigma(X)\) is recursively compact; see also
\cite[Proposition~3.30]{BHLSthermo}.
\end{proof}

A real number is \emph{lower semicomputable} if it is the limit of a computable
nondecreasing sequence of rational numbers.  For \(n\geq1\) and a probability
measure \(\mu\) on \(\Sigma\), put
\[
 H_n(\mu)=-\sum_{w\in\A^n}\mu([w])\log\mu([w]),
 \qquad 0\log0:=0.
\]
The functions \(H_n\) are computable and continuous, uniformly in \(n\), and
for every invariant measure
\[
 h_\mu(\sigma)=\inf_{n\geq1}\frac1nH_n(\mu).
 \tag{2.2}\label{eq:blockformula}
\]

\begin{corollary}[Effective equilibrium-state sets]
\label{thm:equilibrium}
Let \(X\subset\A^{\mathbb Z}\) be a recursively compact subshift, let
\(\varphi:\A^{\mathbb Z}\to\mathbb R\) be computable and continuous, and
suppose that a lower semicomputable name for \(P_X(\varphi)\) is given.  Then
the set of equilibrium states of \(\varphi|_X\) is recursively compact,
uniformly from a recursively compact name for \(X\), a computable name for
\(\varphi\), and the given pressure name.  If the equilibrium state is unique,
it is computable from these names.
\end{corollary}

\begin{proof}
Since \(X\) is recursively compact, the set of probability measures supported
on \(X\) is recursively compact
\cite[Lemma~2.5.1 and Proposition~2.4.4]{GHRinvariant}.  The shift
push-forward is computable \cite[Theorem~3.1.1]{GHRinvariant}, so its fixed
point set
\(K=\M_\sigma(X)\) is recursively compact, uniformly from the given name for
\(X\).  Integration against \(\varphi\) is computable and continuous
\cite[Corollary~4.3.2]{HoyrupRojas}.  Hence the
functions
\[
 F_n^\varphi(\mu)
 =\frac1nH_n(\mu)+\int\varphi\,d\mu
\]
are uniformly computable and continuous, while \eqref{eq:blockformula}
identifies their infimum with free energy.  Put \(s=P_X(\varphi)\).  Since
\(s\) is lower semicomputable, let \(q_k\nearrow s\) be a computable sequence
of rationals.  For every \(n\),
\[
 \{\mu:F_n^\varphi(\mu)<s\}
 =\bigcup_{k\geq1}\{\mu:F_n^\varphi(\mu)<q_k\}.
\]
The right side is a uniformly enumerable union of effectively open sets, so
these strict sublevel sets are lower semicomputable open, uniformly in \(n\),
and
\[
 \Ecal(\varphi)
 =
 K\setminus
 \bigcup_{n\geq1}\{\mu:F_n^\varphi(\mu)<s\}.
\]
If \(\mu\) belongs to the right-hand side, then
\(F_n^\varphi(\mu)\geq s\) for every \(n\).  Its free energy is therefore at
least \(s\), and hence equals \(s\) by the definition of pressure.  Conversely,
every equilibrium state belongs to the right-hand side.  Removing a uniformly
lower semicomputable open set from the recursively compact set \(K\) gives a
recursively compact set.  This is the finite-alphabet instance of the argument
in \cite[Lemma~6.3]{BHLZequilibrium}.  Finally, a recursively compact singleton
is a computable point \cite[Proposition~2.4.1(1)]{GHRinvariant}.
\end{proof}

\section{Finite cores and sequential pressure}
\label{sec:sequential-boundary}

\subsection{Finite-core exhaustion}

Let \(X=X(\G)\), and retain the arbitrary generator enumeration and finite
cores \(X_m\) fixed in the introduction.  We now prove
\Cref{thm:sequential-exhaustion}.  We use the countable-full-shift variational
principle of Burr, Das, Wolf, and Yang
\cite[Theorem~27]{BDWYpressure}; for foundational countable-state
thermodynamic formalism, see \cite{SarigCMS}.  We first state the identities
obtained by inducing on generator boundaries, followed by their
finite-generator counterparts.

\begin{lemma}[Boundary induction]
\label{lem:boundary-coding}
Suppose that \(\G\) uniquely represents \(\Xcon(\G)\), put
\(Y=\G^{\mathbb Z}\), with \(\G\) discrete, and give \(Y\) the full-shift
metric
\[
 d_Y(y,z)=
 \begin{cases}
  0,&y=z,\\
  2^{-\min\{|j|:y_j\neq z_j\}},&y\neq z.
 \end{cases}
\]
Let \(\sigma_Y\) be the shift on generator names.  Write
\[
 c(y)=\cdots y_{-1}.y_0y_1\cdots.
\]
The discrete tower
\[
 \widehat Y=\{(y,j):y\in Y,\ 0\leq j<|y_0|\}
\]
has tower map
\[
 T(y,j)=
 \begin{cases}
  (y,j+1),&j+1<|y_0|,\\
  (\sigma_Yy,0),&j+1=|y_0|.
 \end{cases}
\]
This tower is Borel conjugate to
\((\Xcon(\G),\sigma)\) under concatenation.  In particular,
\(E_{\G}\) is Borel, and the generator-name map
\(\kappa:E_{\G}\to Y\) is a Borel conjugacy from its first-return map to
\((Y,\sigma_Y)\).  The first-return time is
\[
 r(y)=|y_0|.
\]

If \(\mu\in\M_\sigma(X)\) and \(\mu(\Xcon(\G))=1\), then
\(\mu(E_{\G})>0\).  Write
\[
 \mu_E=\frac{\mu|_{E_{\G}}}{\mu(E_{\G})},
 \qquad \eta=\kappa_*\mu_E.
\]
Then \(\eta\) is an invariant probability on \(Y\) with
\[
 R_\eta:=\int r\,d\eta=\frac1{\mu(E_{\G})}<\infty.
\]
For every \(\varphi\in C(X)\), define
\[
 \widetilde\varphi(y)
 =\sum_{j=0}^{|y_0|-1}\varphi(\sigma^jc(y)).
\]
Then
\[
 h_\mu(\sigma)=\frac{h_\eta(\sigma_Y)}{R_\eta},
 \qquad
 \int\varphi\,d\mu
 =\frac{\int\widetilde\varphi\,d\eta}{R_\eta}.
\]
\end{lemma}

\begin{proof}
The tower \(\widehat Y\) is a standard Borel space.  Concatenation defines a
Borel map
\[
 \pi(y,j)=\sigma^jc(y)
\]
onto \(\Xcon(\G)\).  Indeed, since every generator has positive length, every
fixed finite block of \(\pi(y,j)\) is determined by \(j\) and finitely many
coordinates of \(y\).  Thus \(\pi\) is Borel, and it is continuous on each
tower level.  It is injective precisely because pointed
representations are unique in the sense of
\Cref{def:unique-representation}.  The Lusin--Souslin theorem
\cite[Theorem~15.1]{Kechris} therefore shows
that \(\Xcon(\G)=\pi(\widehat Y)\) is Borel and that
\(\pi^{-1}:\Xcon(\G)\to\widehat Y\) is Borel.  Moreover,
\(\pi\circ T=\sigma\circ\pi\).
The zero level maps to \(E_{\G}\), so the same theorem shows explicitly that
\(E_{\G}\) is Borel.  The inverse of that zero-level map is the generator-name
map \(\kappa\), and it identifies the return map and return time.

We have
\[
 \Xcon(\G)=\bigcup_{n\in\mathbb Z}\sigma^nE_{\G}.
\]
Hence an invariant
probability giving full mass to \(\Xcon(\G)\) must give positive mass to
\(E_{\G}\).  Put \(E=E_{\G}\), and let \(\sigma_E\) and \(\tau_E\) denote the
first-return map and return time.  The normalized restriction
\(\mu_E=\mu|_E/\mu(E)\) is \(\sigma_E\)-invariant, and
\(\eta=\kappa_*\mu_E\) is therefore \(\sigma_Y\)-invariant.  Since
\(\tau_E=r\circ\kappa\), Kac's formula \cite{Kac} gives
\[
 R_\eta=\int_E\tau_E\,d\mu_E=\frac1{\mu(E)}.
\]
The map \(\kappa\) conjugates the induced systems, so Abramov's formula
\cite{Abramov} gives
\[
 h_\eta(\sigma_Y)=h_{\mu_E}(\sigma_E)
 =\frac{h_\mu(\sigma)}{\mu(E)}=R_\eta h_\mu(\sigma).
\]
Finally, the return excursions partition \(X\) modulo \(\mu\), and hence
\begin{align*}
 \int\widetilde\varphi\,d\eta
 &=\frac1{\mu(E)}\int_E
     \sum_{j=0}^{\tau_E(x)-1}\varphi(\sigma^jx)\,d\mu(x)\\
 &=\frac1{\mu(E)}\int_X\varphi\,d\mu
 =R_\eta\int_X\varphi\,d\mu.
\end{align*}
Rearranging the last two displays proves the stated identities.
\end{proof}

\begin{lemma}[Finite-code suspension identity]
\label{lem:finite-code-lift}
Suppose that \(\G\) uniquely represents \(\Xcon(\G)\).  Let
\(\varnothing\neq F\subset\G\) be finite, let \(\sigma_F\) be the shift on
\(F^{\mathbb Z}\), and let \(\eta\) be a \(\sigma_F\)-invariant probability.
Write \(c(y)=\cdots y_{-1}.y_0y_1\cdots\), put
\(r(y)=|y_0|\) and \(R_\eta=\int r\,d\eta\).  Since \(F\) is finite,
\[
 1\leq R_\eta\leq\max_{g\in F}|g|<\infty.
\]
For \(\varphi\in C(X)\), define
\[
 \widetilde\varphi(y)
 =\sum_{j=0}^{|y_0|-1}\varphi(\sigma^jc(y)).
\]
The normalized suspension lift \(\widehat\mu\) of \(\eta\) is an invariant
probability measure on \(X(F)\), and
\[
 h_{\widehat\mu}(\sigma)=\frac{h_\eta(\sigma_F)}{R_\eta},
 \qquad
 \int\varphi\,d\widehat\mu
   =\frac{\int\widetilde\varphi\,d\eta}{R_\eta}.
\]
Consequently, for all \(q\in\mathbb R\),
\[
 h_\eta(\sigma_F)+\int(\widetilde\varphi-qr)\,d\eta
 =R_\eta\left(h_{\widehat\mu}(\sigma)
              +\int\varphi\,d\widehat\mu-q\right).
\]
\end{lemma}

\begin{proof}
Restrict the tower in \Cref{lem:boundary-coding} to
\[
 \widehat Y_F=\{(y,j):y\in F^{\mathbb Z},\ 0\leq j<r(y)\}
\]
and give it the measure
\[
 \overline\eta(A)=\frac1{R_\eta}
 \int\sum_{j=0}^{r(y)-1}\mathbf1_A(y,j)\,d\eta(y).
\]
First,
\(\overline\eta(\widehat Y_F)=R_\eta^{-1}\int r\,d\eta=1\).  Moreover, for every
bounded measurable \(f\), invariance of \(\eta\) gives
\begin{align*}
 \int f\circ T\,d\overline\eta
 &=\frac1{R_\eta}\int\left(
     \sum_{j=1}^{r(y)-1}f(y,j)+f(\sigma_Fy,0)\right)d\eta(y)\\
 &=\frac1{R_\eta}\int\sum_{j=0}^{r(y)-1}f(y,j)\,d\eta(y)
 =\int f\,d\overline\eta.
\end{align*}
Thus \(\overline\eta\) is \(T\)-invariant.  Put
\(\widehat\mu=\pi_*\overline\eta\), where \(\pi\) is the concatenation map.
Unique representation makes \(\pi\) a measure-theoretic isomorphism between
the suspension and its image.  The zero level has
\(\overline\eta\)-measure \(1/R_\eta\), its induced map is \(\sigma_F\), and
its normalized induced measure is \(\eta\).  Abramov's formula therefore gives
\[
 h_{\widehat\mu}(\sigma)=h_{\overline\eta}(T)
 =\frac{h_\eta(\sigma_F)}{R_\eta}.
\]
The definition of \(\overline\eta\) also gives
\[
 \int\varphi\,d\widehat\mu
 =\int\varphi\circ\pi\,d\overline\eta
 =\frac1{R_\eta}\int\widetilde\varphi\,d\eta.
\]
The last display in the statement is obtained by subtracting
\(q\int r\,d\eta=qR_\eta\).
\end{proof}

\begin{proof}[Proof of \Cref{thm:sequential-exhaustion}]
For every nonempty finite \(F\subset\G\), the concatenation set \(\Xcon(F)\) is
compact.  Indeed, the generator-name shift \(F^{\mathbb Z}\) is compact, the
roof is bounded, and concatenation from the corresponding finite tower is
continuous.  Hence
\[
 X(F)=\Xcon(F)\subset\Xcon(\G).
\]
If \(\G\) is finite, it follows that \(X=\Xcon(\G)\), so all invariant
measures are sequential.  Moreover, \(X_m=X\) for all sufficiently large
\(m\), proving the result in this case.  We may therefore assume that \(\G\)
is infinite.  The same compactness argument gives
\(X_m\subset\Xcon(\G)\), and
therefore
\[
 \sup_mP_{X_m}(\varphi)\leq P_{\rm seq}(\varphi,\G).
 \tag{3.1}\label{eq:easy-sequential-bound}
\]

Assume first that \(\varphi\) is locally constant.  Fix an ergodic invariant
measure \(\mu\) satisfying \(\mu(\Xcon(\G))=1\), and put
\[
 q=h_\mu(\sigma)+\int\varphi\,d\mu.
\]
Let \(E=E_{\G}\).  By \Cref{lem:boundary-coding}, the generator-name map
\(\kappa:E\to Y\)
measure-theoretically conjugates the induced system on \(E\) to the full shift
\(Y=\G^{\mathbb Z}\), and \(\mu(E)>0\).  Let
\[
 \nu=\kappa_*\left(\frac{\mu|_E}{\mu(E)}\right)
\]
be the induced generator-name measure, write \(r(y)=|y_0|\), and let
\[
 \widetilde\varphi(y)
 =\sum_{j=0}^{r(y)-1}\varphi(\sigma^jc(y))
\]
be the induced potential.  The identities in \Cref{lem:boundary-coding} give
\begin{align}
 R_\nu:=\int r\,d\nu&=\frac1{\mu(E)}<\infty,
 \notag\\
 h_\nu(\sigma_Y)+\int\widetilde\varphi\,d\nu
 &=R_\nu q.
 \tag{3.2}\label{eq:induced-free-energy}
\end{align}
Following the normalization used in
\cite[Proposition~28]{BDWYpressure}, but taking the normalization level to be
the free energy of the present sequential measure, set
\[
 \psi=\widetilde\varphi-qr.
\]
Then
\[
 h_\nu(\sigma_Y)+\int\psi\,d\nu=0.
 \tag{3.3}\label{eq:normalized-induced-zero}
\]
Identify \(g_i\) with the symbol \(i\), put \(Y_m=F_m^{\mathbb Z}\), and, for
\(g\in\G\), let \([g]_Y=\{y\in Y:y_0=g\}\).  These are exactly the
finite-alphabet shifts in
\cite[Theorem~27]{BDWYpressure}; no ordering by generator length is used.  The
countable-full-shift variational principle in that theorem, with \(P_Y\)
denoting the countable-shift pressure, states that a uniformly continuous
potential \(\theta:Y\to\mathbb R\) with
\[
 \operatorname{osc}_1(\theta)
 =\sup_{g\in\G}\left(
     \sup_{[g]_Y}\theta-\inf_{[g]_Y}\theta\right)<\infty
\]
satisfies
\[
 \sup_m P_{Y_m}(\theta)=P_Y(\theta)
 \geq h_\xi(\sigma_Y)+\int\theta\,d\xi
\]
for every invariant probability \(\xi\) for which \(\theta\) is integrable.
We now check these hypotheses for \(\psi\).  Suppose that \(\varphi\) depends
on the coordinates \([-k,k]\).  If \(y,z\in Y\) agree on generator
coordinates \([-k,k]\), then \(\widetilde\varphi(y)=\widetilde\varphi(z)\):
because every generator is nonempty, a letter window of radius \(k\) crosses
at most \(k\) generator boundaries in either direction.
Thus \(\widetilde\varphi\) is locally constant, and hence uniformly
continuous.  On \([g]_Y\), every summand whose \([-k,k]\)-window lies inside
\(g\) is fixed.  Only the summands within \(k\) coordinates of the two ends can
vary, so
\[
 \operatorname{osc}(\widetilde\varphi|_{[g]_Y})
 \leq 4k\|\varphi\|_\infty
\]
uniformly in \(g\), with the same bound when \(|g|<2k\).  The roof \(r\) is
constant on each cylinder \([g]_Y\); indeed, \(d_Y(y,z)<1\) implies
\(y_0=z_0\) and hence \(r(y)=r(z)\).  Therefore
\(\psi=\widetilde\varphi-qr\) is uniformly continuous and has finite
one-cylinder oscillation.  Finally,
\[
 |\widetilde\varphi-qr|
 \leq(\|\varphi\|_\infty+|q|)r,
\]
and \(\int r\,d\nu<\infty\), so \(\psi\) is integrable at \(\nu\).  The theorem
and \eqref{eq:normalized-induced-zero} therefore give
\[
 \sup_mP_{Y_m}(\psi)=P_Y(\psi)
 \geq h_\nu(\sigma_Y)+\int\psi\,d\nu=0.
\]
For every \(\eps>0\), choose \(m\) such that
\(P_{Y_m}(\psi)>-\eps/2\).  By the finite-alphabet variational principle,
choose an invariant probability measure \(\eta_m\) on \(Y_m\) such that
\[
 h_{\eta_m}(\sigma_Y)+\int\psi\,d\eta_m
 >P_{Y_m}(\psi)-\eps/2>-\eps.
 \tag{3.4}\label{eq:finite-induced-approximation}
\]
Because the generator alphabet of \(Y_m\) is finite,
\(R_m:=\int r\,d\eta_m<\infty\).  Lift \(\eta_m\) through the finite code to
an invariant probability measure \(\widehat\mu_m\) on \(X_m\).
\Cref{lem:finite-code-lift} gives
\[
 h_{\eta_m}(\sigma_Y)+\int\psi\,d\eta_m
 =R_m\left(h_{\widehat\mu_m}(\sigma)
       +\int\varphi\,d\widehat\mu_m-q\right).
 \tag{3.5}\label{eq:lifted-normalized-energy}
\]
Since every generator has positive length, \(R_m\geq1\).  Equations
\eqref{eq:finite-induced-approximation}--\eqref{eq:lifted-normalized-energy}
therefore imply
\[
 h_{\widehat\mu_m}(\sigma)+\int\varphi\,d\widehat\mu_m-q
 >-\frac{\eps}{R_m}\geq-\eps.
\]
Consequently,
\[
 P_{X_m}(\varphi)
 \geq h_{\widehat\mu_m}(\sigma)+\int\varphi\,d\widehat\mu_m
 >q-\eps.
\]
Let \(\tau\) be the ergodic decomposition of a sequential measure.  Since
\(\Xcon(\G)\) is invariant and Borel,
\[
 1=\int \xi(\Xcon(\G))\,d\tau(\xi),
\]
so almost every component \(\xi\) is sequential.  Entropy and integration are
affine under ergodic decomposition, so free energy is the \(\tau\)-average of
the component free energies.  Consequently, the supremum defining sequential
pressure may be taken over ergodic sequential measures.  Taking the supremum
over \(\mu\) proves the reverse of \eqref{eq:easy-sequential-bound} for locally
constant potentials.

Finally, locally constant potentials are uniformly dense in \(C(X)\).  Both
\(\varphi\mapsto P_{\rm seq}(\varphi,\G)\) and
\(\varphi\mapsto\sup_mP_{X_m}(\varphi)\) are \(1\)-Lipschitz in the uniform
norm: for each invariant measure, changing the potential from \(\varphi_1\) to
\(\varphi_2\) changes free energy by at most
\(\|\varphi_1-\varphi_2\|_\infty\), and taking the relevant suprema preserves
this bound.  Thus the equality extends to every continuous potential.
Monotonicity of \(X_m\) gives the limit.
\end{proof}

\subsection{FSP is strictly larger than FSSP}
\label{sec:dyck-separation}

\begin{proof}[Proof of \Cref{thm:dyck-separation}]
The construction is inspired by the usual Dyck presentation; see
\cite[Example~4]{KSWergodic} and \cite[Example~5.5]{Pavlovcoded}.  Let
\(\A=\{+,-\}\), with \(\chi(+)=1\) and
\(\chi(-)=-1\).  For a word \(w=w_0\cdots w_{\ell-1}\), define
\[
 H_j(w)=\sum_{i=0}^{j-1}\chi(w_i)
 \qquad (0\leq j\leq\ell).
\]
We use the following collection as our generating set:
\[
 \G_D=
 \left\{w\in\bigcup_{\ell\geq1}\A^\ell:
 H_{|w|}(w)=0\text{ and }H_j(w)>0
                 \text{ for }1\leq j<|w|\right\}.
 \tag{3.6}\label{eq:Dyck-generators}
\]
Thus each generator returns to height zero at its endpoint and stays strictly
positive in between.  Put \(X=X(\G_D)\).

We first show that \(X\) is the full binary shift.  Let \(u\) be any finite
word over \(\A\), and write \(q_j=H_j(u)\).  Choose an integer \(M\) large
enough that
\[
 M+q_j>0\qquad(0\leq j\leq |u|).
\]
Then
\[
 +^M u\,-^{M+q_{|u|}}\in\G_D.
 \tag{3.7}\label{eq:Dyck-embedding}
\]
Indeed, its height during the copy of \(u\) is \(M+q_j>0\), and along the
final string of minuses it decreases to zero only at the endpoint.  Thus the
word has total height zero and every proper nonempty prefix has positive
height.  Since \(+-\in\G_D\), the generator in
\eqref{eq:Dyck-embedding} can be placed between bi-infinite concatenations of
copies of \(+-\).  It follows that every finite binary word occurs in a point
of \(\Xcon(\G_D)\), and therefore
\[
 X(\G_D)=\A^{\mathbb Z}.
 \tag{3.8}\label{eq:Dyck-full-shift}
\]

We next show that this presentation is unique.  For
\(x\in\A^{\mathbb Z}\), define
\[
 S_x(n)=
 \begin{cases}
  \displaystyle\sum_{i=0}^{n-1}\chi(x_i),&n>0,\\[6pt]
  0,&n=0,\\[4pt]
  \displaystyle-\sum_{i=n}^{-1}\chi(x_i),&n<0.
 \end{cases}
 \tag{3.9}\label{eq:two-sided-height}
\]
Suppose that \(x\in\Xcon(\G_D)\), and consider any decomposition of \(x\)
into generators.  The value of \(S_x\) is the same at every generator
boundary, while it is strictly larger between consecutive boundaries.  If the
common boundary height is \(c\), then \(c=\min_{k\in\mathbb Z}S_x(k)\) and
\[
 \{\text{generator boundaries of }x\}
 =\{n\in\mathbb Z:S_x(n)=c\}.
 \tag{3.10}\label{eq:Dyck-boundaries}
\]
The right-hand side depends only on \(x\), so every decomposition has the same
boundaries.  The words between consecutive boundaries and the position of the
origin are then determined.  Therefore \(\G_D\) uniquely represents
\(\Xcon(\G_D)\).

We next compute the sequential entropy.  For \(n\geq1\), let
\(D_n=\{g\in\G_D:|g|=2n\}\).  This set has Catalan cardinality
\[
 |D_n|=C_{n-1}=\frac1n\binom{2n-2}{n-1}.
\]
Indeed, deleting the first \(+\) and last \(-\) gives a bijection from \(D_n\)
to the Dyck words of semilength \(n-1\).
Let \(\eta_n\) be the uniform Bernoulli measure on \(D_n^{\mathbb Z}\), and let
\(\mu_n\) be its normalized suspension under concatenation.  The roof is the
constant \(2n\), so \(\mu_n\) is an invariant sequential measure and Abramov's
formula gives
\[
 h_{\mu_n}(\sigma)=\frac{\log C_{n-1}}{2n}\longrightarrow\log2.
 \tag{3.11}\label{eq:Dyck-Catalan-entropy}
\]
Here the limit follows, for example, from
\(C_{n-1}\sim4^{n-1}/(\sqrt\pi\,(n-1)^{3/2})\).
Since \(X(\G_D)\) is the full binary shift, no invariant measure has entropy
larger than \(\log2\).  Consequently
\[
 P_{\rm seq}(0,\G_D)=\log2=P_X(0).
 \tag{3.12}\label{eq:Dyck-full-sequential-entropy}
\]

It remains to compute the residual pressure.  The height function \(S_x\) is
bounded below for every \(x\in\Xcon(\G_D)\), since its minimum is the common
height at the generator boundaries.  Let \(\nu\) be an ergodic invariant
measure and set
\[
 a_\nu=\int\chi(x_0)\,d\nu(x).
\]
The bilateral ergodic theorem gives, for \(\nu\)-almost every \(x\),
\[
 \frac{S_x(n)}n\longrightarrow a_\nu,
 \qquad
 \frac{S_x(-n)}n\longrightarrow-a_\nu
 \quad(n\to\infty).
 \tag{3.13}\label{eq:Dyck-drift}
\]
If \(a_\nu\neq0\), then the height is unbounded below in one of the two
directions.  Such a point cannot belong to \(\Xcon(\G_D)\), and hence
\(\nu(\Xres(\G_D))=1\).

Fix \(\varphi\in C(X)\).  Since the entropy map is upper semicontinuous on a
finite-alphabet subshift, \(\varphi\) has an equilibrium state.  Free energy is
affine under ergodic decomposition and every component has free energy at most
\(P_X(\varphi)\).  Since the component free energies average to
\(P_X(\varphi)\), almost every ergodic component is again an equilibrium state.
We may therefore choose an ergodic equilibrium state \(\mu\).  If
\(a_\mu\neq0\), then \(\mu\) is residual by
\eqref{eq:Dyck-drift}.  It follows that
\[
 P_{\rm res}(\varphi,\G_D)
 \geq h_\mu(\sigma)+\int\varphi\,d\mu=P_X(\varphi).
\]
The reverse inequality follows from the variational principle, so in this
case \(P_{\rm res}(\varphi,\G_D)=P_X(\varphi)\).

Now suppose that \(a_\mu=0\).  For \(0<\delta<1\), define \(\nu_\delta\) as
follows.  First sample \(x\) according to \(\mu\).  Independently at each
coordinate, change \(x_i=-\) to \(+\) with probability \(\delta\), and leave
every other symbol unchanged.  The resulting measure is a factor of the
product of \(\mu\) with a Bernoulli measure.  Since a product of an ergodic
system with a mixing Bernoulli system is ergodic, \(\nu_\delta\) is ergodic.
The identity \(a_\mu=0\) gives \(\mu[+]=\mu[-]=1/2\), and hence
\[
 a_{\nu_\delta}=a_\mu+2\delta\mu[-]=\delta,
\]
so \(\nu_\delta\) is residual by \eqref{eq:Dyck-drift}.  Moreover,
\(\nu_\delta\to\mu\) in the weak-star topology as \(\delta\downarrow0\): for
a cylinder depending on \(L\) coordinates, the coupling changes at least one
of those coordinates with probability at most \(L\delta/2\).

We also lose arbitrarily little entropy in this perturbation.  Write \(H\) for
Shannon entropy, let \(E_i\)
indicate whether the input symbol at coordinate \(i\) was changed, and write
\(Z\) and \(Y\) for the input and output processes.  The input block is
determined by the output block and the corresponding error block, and
\(\mathbb P(E_i=1)=\delta/2\).  Therefore, for every \(N\),
\[
 H(Z_0^{N-1})
 \leq H(Y_0^{N-1})+H(E_0^{N-1})
 \leq H(Y_0^{N-1})
       +N H_2(\delta/2),
 \tag{3.14}\label{eq:Dyck-entropy-coupling}
\]
where \(H_2(t)=-t\log t-(1-t)\log(1-t)\).  The second inequality follows from
subadditivity:
\[
 H(E_0^{N-1})\leq\sum_{i=0}^{N-1}H(E_i)=NH_2(\delta/2).
\]
The error indicators need not be independent.  Dividing by \(N\) and passing
to entropy rates gives
\[
 h_{\nu_\delta}(\sigma)
 \geq h_\mu(\sigma)-H_2(\delta/2).
\]
Together with weak-star convergence, this yields
\[
 \liminf_{\delta\downarrow0}
 \left(h_{\nu_\delta}(\sigma)+\int\varphi\,d\nu_\delta\right)
 \geq h_\mu(\sigma)+\int\varphi\,d\mu=P_X(\varphi).
\]
Thus residual measures have free energies arbitrarily close to
\(P_X(\varphi)\).  The opposite inequality follows from the variational
principle, and we conclude that
\[
 P_{\rm res}(\varphi,\G_D)=P_X(\varphi)
 \qquad(\varphi\in C(X)).
 \tag{3.15}\label{eq:Dyck-universal-residual-pressure}
\]

It follows that \(P_{\rm seq}>P_{\rm res}\) never occurs for this
presentation.  Therefore
\[
 \operatorname{SSP}(X,\G_D)
 =\operatorname{FSSP}(X,\G_D)=\varnothing.
\]
On the other hand, \eqref{eq:Dyck-full-sequential-entropy} gives
\(0\in\operatorname{FSP}(X,\G_D)\), proving the claimed separation.

Finally, the fair Bernoulli measure is the unique measure of maximal entropy
of the full binary shift.  Its height function is a two-sided simple random
walk and, by recurrence, is unbounded below almost surely, so this measure is
residual.  Since it is the unique measure with entropy \(\log2\), no sequential
measure attains
the sequential entropy in \eqref{eq:Dyck-full-sequential-entropy}.  The
presentation is also effective: one can enumerate \(\G_D\) by testing finite
words, and the language of \(X\) is the full binary language.
\end{proof}

\section{Computability and consequences}
\label{sec:effective-consequences}

\subsection{Pressure and equilibrium states}

\begin{proposition}[Pressure computation from finite cores]
\label{prop:pressure-bracketing}
Let \(X=X(\G)\) be a coded shift over a finite alphabet.  Suppose that \(\G\)
uniquely represents \(\Xcon(\G)\), and that a generator oracle and a
length-ordered language oracle are given.  Let
\(\varphi:\A^{\mathbb Z}\to\mathbb R\) be computable and continuous, and
assume that \(\varphi|_X\in\operatorname{FSP}(X,\G)\).  Then
\(P_X(\varphi)\) is computable uniformly from the two oracles and a computable
name for \(\varphi\).
\end{proposition}

\begin{proof}
Put \(P=P_X(\varphi)\).  From the language oracle and a computable name for
\(\varphi\), the method of Burr, Das, Wolf, and Yang
\cite[Proposition~25]{BDWYpressure} gives a computable sequence of rational
upper bounds
\[
 U_s\searrow P.
\]
For each \(m\), the finite prefix \(F_m\) of the generator oracle effectively
presents the sofic shift \(X_m\), and \(P_{X_m}(\varphi)\) is computable
uniformly from that prefix and the name for \(\varphi\)
\cite[Corollary~24]{BDWYpressure}.  In particular, we can compute rationals
\(\ell_{m,t}\) such that
\[
 \ell_{m,t}\leq P_{X_m}(\varphi)<\ell_{m,t}+2^{-t}.
\]
Dovetail these computations and set
\[
 L_s=\max\{\ell_{m,t}:1\leq m,t\leq s\}.
\]
This gives a computable sequence
\[
 L_s\nearrow\sup_{m\geq1}P_{X_m}(\varphi).
\]
By \Cref{thm:sequential-exhaustion} and the FSP hypothesis, the latter
supremum is \(P_{\rm seq}(\varphi,\G)=P\).  Thus \(L_s\leq P\leq U_s\), and
both bounds converge to \(P\).  Given \(k\), run the two computations until
\(U_s-L_s<2^{-k}\), and output any rational number in \([L_s,U_s]\).  This
search terminates because both sequences converge to \(P\).
\end{proof}

\begin{proof}[Proof of \Cref{cor:pressure-computability}]
Pressure computability follows from \Cref{prop:pressure-bracketing}.  The
language oracle computes a recursively compact name for \(X\) by
\Cref{lem:effective-presentation}, so the conclusions about equilibrium states
follow from \Cref{thm:equilibrium}.
\end{proof}

\subsection{Equality examples}
\label{sec:equality-examples}

Both examples use the code from \cite[Example~5.4]{Pavlovcoded}.  Let
\(B=\{1,2,3,4\}\), let \(\A=\{0,1,2,3,4\}\), put
\(k(n)=\lfloor\log_2n\rfloor\), and define
\[
 \G=\{w0^{k(n)}:n\geq2,\ w\in B^n\}.
 \tag{4.1}\label{eq:Pavlov-Markov-code}
\]
Write \(X=X(\G)\).

\begin{lemma}[Effective presentation of the Pavlov code]
\label{lem:Pavlov-code-effective}
The collection \(\G\) uniquely represents \(\Xcon(\G)\), and both its
generator oracle and the length-ordered language oracle for \(X\) are
computable.
\end{lemma}

\begin{proof}
Within every point of \(\Xcon(\G)\), the transitions \(0B\) are exactly the
generator boundaries.  The boundary locations, the words between consecutive
boundaries, and the position of the origin are therefore determined by the
point.  This proves unique representation.  To compute the generator oracle,
list the finitely many words \(w0^{k(n)}\), with \(w\in B^n\), one value of
\(n\geq2\) at a time.

We now show that membership in each finite language level is decidable.  Fix a
word \(u\) of length \(L\).  If \(u\in\Lcal(X)\), then the cylinder defined by
\(u\) meets \(X=\overline{\Xcon(\G)}\).  Since the cylinder is open, it also
meets \(\Xcon(\G)\), and the relevant finite window is contained in a finite
concatenation of generators.  The converse is immediate.  Any such occurrence
intersects at most \(L+2\) generators: at most two are only partially visible,
and every generator lying entirely inside the occurrence occupies at least one
of its \(L\) coordinates.  A generator lying entirely inside the occurrence
has positive-block length at most \(L\), so only the first and last intersected
generators can have unbounded length.

We claim that the endpoint generators can also be replaced by generators whose
positive-block lengths are at most \(2^{L+1}\), without changing the occurrence
of \(u\).  First suppose that the occurrence intersects more than one
generator.  The visible suffix of the left endpoint is either \(0^t\), or
\(v0^t\) with \(v\in B^a\), in which case the entire zero block is visible.
Align its replacement at the right end.  In the first case \(t\geq1\), and we
take \(n'=2^t\); in the second, take
\[
 n'=\max\{2^t,a\}.
\]
The original generator has positive-block length less than \(2^{t+1}\), so
this choice still satisfies \(\lfloor\log_2n'\rfloor=t\).  In either case, the
visible positive symbols can be extended arbitrarily within \(B\).

The visible prefix of the right endpoint is either a word \(v\in B^a\), or
\(v0^t\), in which case the entire positive block is visible.  Align its
replacement at the left end.  In the first case take
\(n'=\max\{2,a\}\) and extend the positive word arbitrarily.  In the second,
the original positive-block length is at most \(L\), so no replacement is
needed.

It remains to consider an occurrence contained in a single generator.  If it
contains only positive symbols, take \(n'=\max\{2,L\}\).  If it contains only
zeros, take \(n'=2^L\).  Finally, if it crosses the positive-to-zero transition
and displays \(a\) positive symbols followed by \(t\) zeros, take
\[
 n'=\max\{2,a,2^t\}
\]
and align the replacement at that transition.  Each replacement described
above has positive-block length at most \(2^{L+1}\).

Consequently, \(u\) is allowed if and only if it occurs in a concatenation of
at most \(L+2\) generators whose positive-block lengths are at most
\(2^{L+1}\).  There are finitely many such concatenations, so enumerating them
and every possible starting offset gives a finite search deciding membership
in \(\Lcal_L(X)\).  Running this decision procedure one length at a time gives
the ordered language oracle.
\end{proof}

\begin{example}[Residual MME at equality]
\label{ex:pavlov-critical}
Pavlov's analysis
\cite[Example~5.4 and Proposition~5.1]{Pavlovcoded} shows that the
concatenation entropy of this presentation is \(\log4\) and that null
recurrence prevents this entropy from being attained by a sequential measure.
Every generator contains a zero, while every finite \(B\)-word occurs inside
an arbitrarily long positive block, so \(B^{\mathbb Z}\subset X\) and
\(B^{\mathbb Z}\cap\Xcon(\G)=\varnothing\).  His description of the limit set
also shows that the unique global MME is the residual uniform Bernoulli measure
on \(B^{\mathbb Z}\).  Consequently,
\[
 \hcon(X,\G)=\hres(X,\G)=\htop(X)=\log4.
\]
By \Cref{lem:Pavlov-code-effective}, \Cref{cor:main-coded} computes the unique
MME, even though no sequential measure realizes the common entropy.
\end{example}

\begin{example}[A residual Markov equilibrium at equality]
\label{ex:critical-Markov}
We now give a nonzero-potential version of the preceding phenomenon on the
same coded shift.  Its equilibrium state will be a correlated Markov measure
rather than the uniform MME.

The coded shift is not the full shift on \(\A\): for example, \(010\) is
forbidden because every positive block between zero blocks has length at least
two.  On the other hand,
\[
 B^{\mathbb Z}\subset X,
\]
since every finite word over \(B\) occurs arbitrarily far from the zeros in a
generator with a long positive block.  No point of \(B^{\mathbb Z}\) is a
bi-infinite concatenation of generators, so
\(B^{\mathbb Z}\subset\Xres(\G)\).  Thus \(B^{\mathbb Z}\) is a proper
residual subsystem of \(X\).  The zero blocks in the generators have unbounded
length, so \(0^\infty\in X\); it has no generator decomposition and is
therefore residual as well.

Consider the strictly positive stochastic matrix
\[
 P_{ij}=
 \begin{cases}
  1/2,&i=j,\\
  1/6,&i\neq j,
 \end{cases}
 \qquad i,j\in B,
\]
and the computable nearest-neighbor potential
\[
 \varphi(x)=
 \begin{cases}
  \log(4P_{x_0x_1}),&x_0,x_1\in B,\\
  0,&0\in\{x_0,x_1\}.
 \end{cases}
 \tag{4.2}\label{eq:critical-Markov-potential}
\]
Thus \(\varphi=\log2\) on equal positive pairs and
\(\varphi=\log(2/3)\) on unequal positive pairs.  We will prove that
\[
 P_{\rm seq}(\varphi,\G)
 =P_{\rm res}(\varphi,\G)
 =P_X(\varphi)=\log4,
 \tag{4.3}\label{eq:critical-Markov-tie}
\]
that no sequential measure attains this value, and that the unique global
equilibrium state is residual.

First consider the residual side.  Write
\([0B]=\bigcup_{b\in B}[0b]\) and \([B0]=\bigcup_{b\in B}[b0]\).  If an
ergodic residual measure \(\mu\) gave positive mass to \([0B]\), the bilateral
Birkhoff theorem would give \(\mu\)-almost every point infinitely many
occurrences of \(0B\) in both directions.  Consider the word between two
consecutive occurrences.  This finite word occurs in an actual concatenation
because every word in
\(\Lcal(X)\) occurs in \(\Xcon(\G)\).  Within a concatenation, the transitions
\(0B\) are exactly the generator boundaries.  Since the two occurrences were
consecutive, the word between them is a single generator.  Applying this in
both directions gives a bi-infinite generator decomposition, contradicting
residuality.  It follows that every ergodic residual measure has
\(\mu([0B])=0\).  Indeed, invariance gives the flux identity
\[
 \mu([0B])=\mu([0])-\mu([00])
 =\mu(\sigma^{-1}[0])-\mu([00])=\mu([B0]).
\]
Thus \(\mu([B0])=0\) as well, so the indicator of membership in \(B\) is almost
surely constant along an orbit.  The measure is therefore supported either on
\(B^{\mathbb Z}\) or on the fixed point \(0^\infty\).  The same classification
follows from the description of the limit set in
\cite[Example~5.4 and Lemma~4.1]{Pavlovcoded}.  On \(B^{\mathbb Z}\), the
transfer matrix is \(4P\).  The matrix \(P\) is positive and doubly stochastic,
so \(4P\) has Perron eigenvalue \(4\) with constant left and right eigenvectors.
The standard transfer-matrix construction therefore gives transition matrix
\(P\) and the uniform stationary vector, and positivity gives uniqueness.  Thus
the equilibrium state for the restricted system on \(B^{\mathbb Z}\) is the
stationary Markov measure \(\mu_P\).  Since every entry of \(P\) is positive,
\(\mu_P\) has full support on the residual subsystem
\(B^{\mathbb Z}\), and therefore does not have full support on \(X\).  Directly,
\[
 h_{\mu_P}(\sigma)+\int\varphi\,d\mu_P
 =-\sum_{i,j}\tfrac14P_{ij}\log P_{ij}
  +\sum_{i,j}\tfrac14P_{ij}\log(4P_{ij})
 =\log4.
\]
The fixed point \(0^\infty\) has free energy zero.  By ergodic decomposition,
the free energy of any residual measure is the average of the free energies of
the ergodic components just classified.  Perron--Frobenius uniqueness on
\(B^{\mathbb Z}\) therefore makes equality possible only for \(\mu_P\).  Hence
\(P_{\rm res}(\varphi,\G)=\log4\), uniquely attained by \(\mu_P\).  This
measure is not a product measure, since
\(\mu_P([11])=1/8\neq1/16=\mu_P([1])^2\).

For the sequential side, put \(q=\log4\).  If
\(g=w_1\cdots w_n0^{k(n)}\), then the sum of \(\varphi\) over the coordinates
of this generator in any concatenation is independent of the adjacent
generators.  Denote this sum by
\[
 \Phi(g)=\sum_{i=1}^{n-1}\log(4P_{w_iw_{i+1}}).
\]
Since \(P\) is stochastic,
\[
 \sum_{w\in B^n}e^{\Phi(w0^{k(n)})}
 =\boldsymbol1^{\mathsf T}(4P)^{n-1}\boldsymbol1=4^n.
\]
The last equality follows from
\((4P)\boldsymbol1=4\boldsymbol1\) and
\(\boldsymbol1^{\mathsf T}\boldsymbol1=4\).
Consequently the normalized generator weights
\[
 p_g=e^{\Phi(g)-q|g|}
\]
satisfy
\[
 \sum_{g\in\G}p_g
 =\sum_{n=2}^\infty4^{-k(n)}
 =\sum_{j=1}^\infty2^j4^{-j}=1,
 \tag{4.4}\label{eq:critical-Markov-normalization}
\]
where the middle sum groups the \emph{total} weight of all generators with a
fixed positive-block length \(n\).  On the other hand,
\[
 \sum_{g\in\G}|g|p_g
 =\sum_{j=1}^\infty4^{-j}
   \sum_{n=2^j}^{2^{j+1}-1}(n+j)=\infty.
 \tag{4.5}\label{eq:critical-Markov-infinite-mean}
\]
Indeed, the \(j\)-th summand is at least
\[
 4^{-j}\,2^j\,2^j=1.
\]

Let \(\mu\) be an ergodic sequential measure and let \(\eta\) be its induced
generator-name measure from \Cref{lem:boundary-coding}.  Its roof mean
\(R_\eta\) is finite.  If
\(\rho(g)=\eta(\{y\in Y:y_0=g\})\), write
\[
 H(\rho)=-\sum_g\rho(g)\log\rho(g),
 \qquad
 D(\rho\Vert p)=\sum_g\rho(g)\log\frac{\rho(g)}{p_g}.
\]
Let \(N=|g|\) be the generator-length random variable under \(\rho\), and let
\(H(N)\) denote the entropy of its distribution.  For any \(a>0\), Gibbs
inequality applied to this distribution and the geometric law
\(q_a(r)=(1-e^{-a})e^{-a(r-1)}\), \(r\geq1\), gives
\[
 H(N)\leq-\log(1-e^{-a})+a(\mathbb E_\rho N-1)<\infty.
\]
Conditional on \(N=r\), there are at most \(4^r\) generator names, and hence
\[
 H(\rho)=H(N)+H(g\mid N)
 \leq H(N)+(\log4)\mathbb E_\rho N<\infty.
\]
Moreover,
\[
 |\log p_g|=\left|\Phi(g)-q\,|g|\right|
 \leq(\|\varphi\|_\infty+q)|g|,
\]
so \(\sum_g\rho(g)|\log p_g|<\infty\).  Now
\(R_\eta=\sum_g\rho(g)|g|\) and
\(\int\widetilde\varphi\,d\eta=\sum_g\rho(g)\Phi(g)\).  The identities in
\Cref{lem:boundary-coding}, followed by
\(\log p_g=\Phi(g)-q|g|\), therefore give
\begin{align}
 h_\mu(\sigma)+\int\varphi\,d\mu-q
 &=\frac{h_\eta(\sigma_Y)+\int\widetilde\varphi\,d\eta-qR_\eta}
         {R_\eta}\notag\\
 &=\frac{h_\eta(\sigma_Y)+\sum_g\rho(g)\log p_g}{R_\eta}.
 \tag{4.6}\label{eq:critical-Markov-relative-entropy}
\end{align}
The entropy-rate bound and Gibbs inequality now give
\[
 h_\eta(\sigma_Y)+\sum_g\rho(g)\log p_g
 \leq H(\rho)+\sum_g\rho(g)\log p_g
 =-D(\rho\Vert p)\leq0.
\]
If the left-hand side were zero, then both inequalities would be equalities.
In particular, \(D(\rho\Vert p)=0\), so \(\rho=p\).  This is impossible because
\(p\) has infinite mean by \eqref{eq:critical-Markov-infinite-mean}, whereas
\(\rho\) has mean \(R_\eta<\infty\).  Hence every ergodic sequential
measure has free energy strictly below \(q\), and the same is true for every
sequential measure by ergodic decomposition and affinity: equality would
force almost every ergodic component to have free energy \(q\).  To see that
the supremum is still \(q\),
let \(F_\ell\uparrow\G\) be an increasing exhaustion by finite sets, put
\(Z_\ell=\sum_{g\in F_\ell}p_g\), and take the Bernoulli generator-name measure
with marginal \(p_g/Z_\ell\) on \(F_\ell\).
Let \(R_\ell<\infty\) be its roof mean and \(\mu_\ell\) its lift.  Then
\[
 H(p/Z_\ell)+\sum_{g\in F_\ell}\frac{p_g}{Z_\ell}\log p_g
 =\log Z_\ell,
\]
and \(Z_\ell\uparrow1\), while \(R_\ell\geq1\).  Thus
\eqref{eq:critical-Markov-relative-entropy} becomes
\[
 h_{\mu_\ell}(\sigma)+\int\varphi\,d\mu_\ell-q
 =\frac{\log Z_\ell}{R_\ell}\longrightarrow0.
\]
It follows that \(P_{\rm seq}(\varphi,\G)=q\), but the supremum is not attained.
The residual
calculation now proves \eqref{eq:critical-Markov-tie} and shows that \(\mu_P\)
is the unique global equilibrium state.  Indeed, every invariant measure
decomposes over the invariant Borel partition
\(\Xcon(\G)\sqcup\Xres(\G)\).  If
\(0<\alpha:=\mu(\Xcon(\G))<1\), its free energy is
\[
 \alpha\left(h_{\mu_{\rm seq}}(\sigma)
       +\int\varphi\,d\mu_{\rm seq}\right)
 +(1-\alpha)\left(h_{\mu_{\rm res}}(\sigma)
       +\int\varphi\,d\mu_{\rm res}\right),
\]
where the two measures are the normalized restrictions.  The endpoint cases
are the restrictions themselves.  Thus any positive sequential component
makes the free energy strictly smaller than \(q\), while for \(\alpha=0\)
equality holds only for \(\mu_{\rm res}=\mu_P\).
Since \(\varphi\) is computable,
\Cref{lem:Pavlov-code-effective} and \Cref{cor:pressure-computability} together
show that \(\mu_P\) is computable from the generator and language oracles.

\end{example}

\subsection*{Acknowledgements and AI disclosure}
OpenAI Codex (desktop application; GPT-5.6 Sol) was used to explore proof
strategies, test intermediate arguments, identify relevant literature, and
assist with manuscript revision.  The author personally verified every
adopted mathematical argument and citation and assumes sole responsibility
for the final work.

\subsection*{Competing interests and funding}
The author declares no competing interests.  No funding was received for this
work.

\end{document}